\documentclass[12pt, reqno]{amsart}
\usepackage [latin1]{inputenc}
\usepackage{amscd}
\usepackage{epsfig}
\usepackage{amssymb}
\usepackage{amsmath}
\usepackage{amsthm}
\usepackage[T1]{fontenc}
\usepackage{ae,aecompl}

\usepackage {amsfonts} % fuer IR, IN, etc; \mathfrak (Altdeutsche Symb.) - gehoert NICHT zum Standard von 2e
\usepackage {latexsym} % noetig fuer einige Symbole u.a. \Box und \leadsto
\usepackage {exscale} % damit Summen- und Integralzeichen bei >10pt richtig angezeigt werden
\usepackage{amsmath}

\newcommand{\C} {\ensuremath{\mathbb{C}}}

\newcommand{\OO}{\mathcal{O}}

\newcommand{\dq}{\overline{\partial}}
\newcommand{\wt}[1]{\widetilde{#1}}

\DeclareMathOperator{\Sing}{Sing}

\newtheorem {satz} {Satz} [section]

\newtheorem {cor} [satz] {Korollar}

\newtheorem {thm} [satz] {Theorem}

\DeclareMathOperator{\supp}{supp}

\renewcommand{\theta}{\vartheta}

\title[Hartogs' extension theorem] 
{A $\dq$-theoretical proof of Hartogs' extension theorem on $(n-1)$-complete complex spaces}

\author{J. Ruppenthal}

\address{Department of Mathematics, University of Wuppertal, Gau{\ss}str. 20, 42119 Wuppertal, Germany.}
\email{ruppenthal@uni-wuppertal.de}

\date{November 12, 2008}



\begin{document}

~\\[-3mm]

\begin{abstract} 
Let $X$ be a connected normal complex space of dimension $n\geq 2$ which is $(n-1)$-complete,
and let $\pi: M \rightarrow X$ be a resolution of singularities.
By use of Takegoshi's generalization of the Grauert-Riemenschneider vanishing theorem,
we deduce $H^1_{cpt}(M,\OO)=0$, which in turn implies Hartogs' extension theorem on $X$
by the $\dq$-technique of Ehrenpreis.
\end{abstract}

\maketitle

\section{Introduction}

Whereas first versions of Hartogs' extension theorem (e.g. on polydiscs) were usually obtained
by filling Hartogs' figures with analytic discs, no such geometrical proof was known for the general theorem
in complex number space $\C^n$ for a long time. Proofs of the general theorem in $\C^n$ 
usually depend on the Bochner-Martinelli-Koppelman kernel
or on the solution of the $\dq$-equation with compact support
(the famous idea due to Ehrenpreis \cite{Eh}).

Only recently, J.\ Merker and E.\ Porten were able to fill the gap by giving a geometrical proof 
of Hartogs' extension theorem in $\C^n$ (see \cite{MePo1}) by using a finite number of
parameterized families of holomorphic discs and Morse-theoretical tools for the global topological control 
of monodromy, but no $\dq$-theory or intergal kernels (except the Cauchy kernel).

They also extended their result to the general case of $(n-1)$-complete normal complex spaces (see \cite{MePo2}),
where no proof was known until now at all. One reason is the lack of global integral kernels or an appropriate 
$\dq$-theory for singular complex spaces.
The present paper is an answer to the
question wether it could be possible to use some $\dq$-theoretical considerations
for reproducing the result of Merker and Porten on a $(n-1)$-complete complex space $X$. 
More precisely, we 
solve a  $\dq$-equation with compact support on a desingularization of $X$
in order to derive the following statement by the technique of Ehrenpreis:

\begin{thm}\label{thm:hartogs}
Let $X$ be a connected normal complex space of dimension $n\geq 2$ which is $(n-1)$-complete.
Furthermore, let $D$ be a domain in $X$ and $K\subset D$ a compact subset such that $D\setminus K$ is connected.
Then each holomorphic function $f\in \mathcal{O}(D\setminus K)$ has a unique holomorphic extension to the whole set $D$.
\end{thm}

In the special case of a Stein space $X$ with only isolated singularities,
a $\dq$-theoretical proof of Hartogs' extension theorem was already given in \cite{Rp}.

\vspace{2mm}
The main points of the proof of Theorem \ref{thm:hartogs}
are as follows: Let $M$ be a complex manifold of dimension $n$,
$X$ a complex space and $\pi: M\rightarrow X$ a proper modification.
Then it follows by Takegoshi's generalization
of the Grauert-Riemenschneider vanishing theorem
that
\begin{eqnarray}\label{eq:01}
R^q\pi_* \Omega^n_M =0,\ q>0,
\end{eqnarray}
where $R^q \pi_* \Omega^n_M$ are the higher direct images of
the canonical sheaf $\Omega^n_M$ on $M$ (see Theorem \ref{thm:takegoshi}).
By use of the Leray spectral sequence, \eqref{eq:01} yields
\begin{eqnarray*}
H^q(M,\Omega^n_M) \cong H^q(X,\pi_* \Omega^n_M),
\end{eqnarray*}
and this implies with Serre duality that
$$H^1_{cpt}(M,\OO)\cong H^{n-1}(M,\Omega^n_M) =0$$
if $X$ is $(n-1)$-complete (see Theorem \ref{thm:takegoshi3}).
This is elaborated in section \ref{sec:cohomology},
while we will show in section \ref{sec:proof} that vanishing of $H^1_{cpt}(M,\OO)$ gives
Hartogs' extension theorem on $X$ by the $\dq$-technique of Ehrenpreis,
because the extension problem on $X$ can be reduced to an extension problem on 
the desingularization $M$.
In the last section, we give a few remarks on Takegoshi's vanishing theorem
for convenience of the reader.

\vspace{2mm}
For a more detailed introduction to the topic with a full historical record, remarks
and references, we refer to \cite{MePo1} and \cite{MePo2}.\\

\newpage

\section{Dolbeault Cohomology of Proper Modifications}\label{sec:cohomology}

In its general form,
the Grauert-Riemenschneider vanishing theorem (see \cite{GrRie}, Satz 2.1) states:

\begin{thm}\label{thm:GrRie1}
Let $X$ be an $n$-dimensional compact irreducible reduced complex space with
$n$ independent meromorphic functions (Moishezon), and let $\mathcal{S}$ be
a quasi-positive coherent analytic sheaf without torsion on $X$. Then:
$$H^q(X, \mathcal{S} \otimes \Omega^n_X) = 0, \ \ q>0,$$
where $\Omega^n_X$ is the sheaf of holomorphic $n$-forms on $X$ (the canonical sheaf),
defined in the sense of Grauert and Riemenschneider.
\end{thm}

This generalization of Kodaira's famous vanishing theorem is also proved by means of harmonic theory.
The main point in the proof is %the following 
(\cite{GrRie}, Satz 2.3):

\begin{thm}\label{thm:GrRie2}
Let $X$ be a projective complex space, $\mathcal{S}$ a quasi-positive coherent analytic sheaf
on $X$ without torsion, and let $\pi: M\rightarrow X$ be a resolution of singularities,
such that $\hat{\mathcal{S}}=\mathcal{S}\circ\pi$ is locally free on $M$.
Then:
$$R^q\pi_* ( \hat{\mathcal{S}}\otimes \Omega^n_M) =0, \ \ q>0.$$
\end{thm}

Here, $\mathcal{S}\circ\pi$ denotes the torsion-free preimage sheaf:
$$\mathcal{S}\circ\pi:= \pi^* \mathcal{S} / T(\pi^* \mathcal{S}),$$
where $T(\pi^* \mathcal{S})$ is the coherent torsion sheaf of the preimage $\pi^*\mathcal{S}$
(see \cite{Gr2}, p. 61).
As a simple consequence of Theorem \ref{thm:GrRie2}, one can deduce:

\begin{cor}
Let $M$ be a Moishezon manifold of dimension $n$, and $X$ a projective variety
such that $\pi: M\rightarrow X$ is a resolution of singularities.
Then:
$$R^q \pi_* \Omega^n_M = 0, \ \ q>0,$$
where $R^q \pi_* \Omega^n_M$, $q>0,$ are the higher direct image sheaves of $\Omega^n_M$.
\end{cor}

\begin{proof}
Let $F$ be a positive holomorphic line bundle on $X\subset \C\mathbb{P}^L$,
and $\mathcal{S}$ the sheaf of sections in $F$. So, $S\circ\pi$ is a positive
locally free sheaf on $M$ (\cite{GrRie}, Satz 1.4), and as in \cite{GrRie}, Satz 2.4,
it follows from Theorem \ref{thm:GrRie2} that
$$R^q\pi_* \Omega^n_M \otimes \mathcal{S} = R^q \pi_* (\hat{\mathcal{S}}\otimes\Omega^n_M) =0$$
for $q>0$ which implies the statement.
\end{proof}

As Grauert and Riemenschneider mention already in their original paper \cite{GrRie},
this statement is of local nature and doesn't depend on the projective embedding (whereas their proof does).
And in fact, the result was generalized later
by K. Takegoshi (see \cite{Ta}, Corollary I):

\begin{thm}\label{thm:takegoshi}
Let $M$ be a complex manifold of dimension $n$, and $X$ a complex space such that
$\pi: M\rightarrow X$ is a proper modification. Then:
$$R^q \pi_* \Omega^n_M=0,\ \ q>0.$$
\end{thm}

The nice proof consists mainly of a vanishing theorem on weakly $1$-complete Kähler manifolds 
which is based on $L^2$-estimates for the $\dq$-operator.
%For convenience of the reader, we will give some remarks on the proof in section \ref{sec:remarks}.
As an easy consequence, we deduce:

\begin{thm}\label{thm:takegoshi2}
Let $M$ be a complex manifold of dimension $n$, and $X$ a complex space such that
$\pi: M\rightarrow X$ is a proper modification. Then:
\begin{eqnarray}\label{eq:take1}
H^{n,q}(M) \cong H^q(M,\Omega^n_M) \cong H^q(X,\pi_* \Omega^n_M).
\end{eqnarray}
\end{thm}

\begin{proof}
The proof follows directly by the Leray spectral sequence.
Alternatively, on can use the abstract theorem of deRham:
Let 
$$0 \rightarrow \Omega^n_M \hookrightarrow
\mathcal{C}^\infty_{n,0} \xrightarrow{\ \dq \ }
\mathcal{C}^\infty_{n,1} \xrightarrow{\ \dq \ }
\cdots \xrightarrow{\ \dq \ }
\mathcal{C}^\infty_{n,n} \rightarrow 0$$
be the Dolbeault complex on $M$ which is a fine resolution of $\Omega^n_M$.
Then, it follows by Theorem \ref{thm:takegoshi} that the direct image sequence
\begin{eqnarray}\label{eq:seq2}
0 \rightarrow \pi_* \Omega^n_M \hookrightarrow
\pi_* \mathcal{C}^\infty_{n,0} \xrightarrow{\ \pi_* \dq \ }
\pi_* \mathcal{C}^\infty_{n,1} \xrightarrow{\ \pi_* \dq \ }
\cdots \xrightarrow{\ \pi_* \dq \ }
\pi_* \mathcal{C}^\infty_{n,n} \rightarrow 0
\end{eqnarray}
is again exact because of the properties of the direct image functor, or since
$$\left(R^q\pi_* \Omega^n_M\right)_x= \lim_{\substack{\rightarrow\\ x\in U}} H^q(\pi^{-1}(U),\Omega^n_M).$$

It is clear that the direct image sheaves $\pi_* \mathcal{C}^\infty_{n,q}$ are fine, as well.
So, the sequence \eqref{eq:seq2} is a fine resolution of $\pi_* \Omega^n_M$,
and the abstract Theorem of deRham implies that:
\begin{eqnarray}\label{eq:seq3}
H^q(X,\pi_* \Omega^n_M) \cong \frac{\mbox{Ker}
(\pi_*\dq: \pi_*\mathcal{C}^\infty_{n,q}(X) \rightarrow \pi_*\mathcal{C}^\infty_{n,q+1}(X))}
{\mbox{Im}
(\pi_*\dq: \pi_*\mathcal{C}^\infty_{n,q-1}(X) \rightarrow \pi_*\mathcal{C}^\infty_{n,q}(X))},
\end{eqnarray}
but the right hand side of \eqref{eq:seq3} is nothing else but
\begin{eqnarray*}
\frac{\mbox{Ker}
(\dq: \mathcal{C}^\infty_{n,q}(M) \rightarrow \mathcal{C}^\infty_{n,q+1}(M))}
{\mbox{Im}
(\dq: \mathcal{C}^\infty_{n,q-1}(M) \rightarrow \mathcal{C}^\infty_{n,q}(M))}
= H^{n,q}(M) \cong H^q(M,\Omega^n_M).
\end{eqnarray*}
\end{proof}

Now, if the space $X$ has nice properties,
we can deduce consequences for the Dolbeault cohomology on $M$.
In this paper, we are particularly interested in:

\begin{thm}\label{thm:takegoshi3}
Let $M$ be a complex manifold of dimension $n$, and $X$ a $k$-complete complex space such that
$\pi: M\rightarrow X$ is a proper modification. Then:
\begin{eqnarray}\label{eq:take2}
H^{n-q}_{cpt}(M,\OO) \cong H^q(M,\Omega^n_M) = 0, \ \ k\leq q\leq n.
\end{eqnarray}
\end{thm}

\begin{proof}
Since $X$ is $k$-complete, it follows from Corollary \ref{thm:takegoshi2} that
\begin{eqnarray}\label{eq:vollst}
H^q(M,\Omega^n_M) \cong H(X, \pi_* \Omega^n_M) =0, \ \ k\leq k\leq n,
\end{eqnarray}
because $\pi_* \Omega^n_M$ is coherent by Grauert's direct image theorem (see \cite{Gr2}).
Serre's criterion (\cite{Se}, Proposition 6) tells us that we can apply Serre duality
(\cite{Se}, Th\'eorème 2) to the cohomology groups in \eqref{eq:vollst},
and we get the duality
$$H^{n-q}_{cpt}(M,\OO) \cong H^q(M,\Omega^n_M), \ \ k\leq q\leq n.$$
\end{proof}

%\newpage

\section{Proof of Theorem \ref{thm:hartogs}}\label{sec:proof}

The assumption about normality implies that $X$ is reduced.
Let $$\pi: M \rightarrow X$$ 
be a resolution of singularities,
where $M$ is a complex connected manifold of dimension $n$,
and $\pi$ is a proper holomorphic surjection.
Let $E:=\pi^{-1}(\Sing X)$ be the exceptional set of the desingularization.
Note that
\begin{eqnarray}\label{eq:bih}
\pi|_{M\setminus E}: M\setminus E \rightarrow X\setminus \Sing X
\end{eqnarray}
is a biholomorphic map. For the topic of desingularization we refer to \cite{AHL}, \cite{BiMi} and \cite{Ha}.\\

First, we will observe that the extension problem on $X$ can be reduced to an analogous extension problem
on $M$. Let
$$D':=\pi^{-1}(D),\ K':=\pi^{-1}(K),\ F:=f\circ \pi \in \mathcal{O}(D'\setminus K').$$
Clearly, $D'$ is an open set and $K'$ is compact with $K'\subset D'$
since $\pi$ is a proper holomorphic map.
$D\setminus K$ is a connected normal complex space.
So, it is connected, reduced and locally irreducible, hence 
globally irreducible as well (see \cite{GrRe}).
But then, $D\setminus K\setminus\Sing X$ is still connected.
So, the same is true for  $D'\setminus K'\setminus E$ because of \eqref{eq:bih}. 
But then $D'\setminus K'$ and $D'$ are connected, too.
That means that the assumptions on $D$ and $K$ behave well under desingularization.\\

Now, assume that there exists an extension $\wt{F}\in \OO(D')$ of $F$ such that
$$\wt{F}|_{D'\setminus K'} \equiv F.$$
But then
$$\wt{f} := \wt{F} \circ \big(\pi|_{D'\setminus E}\big)^{-1} \ \in \mathcal{O}(D\setminus \Sing X)$$
equals $f$ on an open subset (of $D\setminus K$), and it has an extension to the whole domain $D$
by Riemann's extension theorem for normal spaces (see \cite{GrRe}, for example).
Hence, 
$$\wt{f}\in \OO(D),\ \ \ \wt{f}|_{D\setminus K}\equiv f,$$
and the extension $\wt{f}$ is unique because $D\setminus K$ is connected and
$X$ is globally and locally irreducible (see again \cite{GrRe}). 
It only remains to show that such an extension $\wt{F}\in\OO(D')$ exists.
But this follows from $H^1_{cpt}(M,\OO)=0$ by Ehrenpreis' $\dq$-technique (see \cite{Eh})
as we will describe in the remainder of this section.\\

Let 
$$\chi\in C^\infty_{cpt}(M)$$
be a smooth cut-off function that is identically one in a neighborhood of $K'$
and has compact support 
$$C:=\supp\chi\subset\subset D'.$$

Consider
$$G:=(1-\chi) F \in C^\infty(D'),$$
which is an extension of $F$ to $D'$, but unfortunately not holomorphic. We have to fix it by
the idea of Ehrenpreis. So, let
$$\omega:=\dq G \in C^\infty_{(0,1),cpt}(D'),$$
which is a $\dq$-closed $(0,1)$-form with compact support in $D'$.
We may consider $\omega$ as a form on $M$ with compact support.
But $H^1_{cpt}(M,\OO)=0$ by Theorem \ref{thm:takegoshi3}.
So, there exists $g\in C^\infty_{cpt}(M)$ such that
$$\dq g=\omega,$$
and $g$ is holomorphic on $M\setminus C$ (where $\omega=\dq G=\dq F=0$).
Let
\begin{eqnarray}\label{eq:wtf}
\wt{F}:=(1-\chi)F-g \in \OO(D').
\end{eqnarray}

\newpage
%~\\[-11mm]

We have to show that $g\equiv 0$ on an open subset of $D'\setminus C$.
To see this, let
$$A:=\supp g.$$
Then, $M\setminus \big(A\cup C\big)\neq \emptyset$ because $A\cup C$ is compact
but $M$ is not. If we would assume that $M$ is compact,
$X$ would be compact as well.
But $X$ is an $(n-1)$-complete space which admitts an $(n-1)$-convex exhaustion function
which cannot have a local maximum. So, $X$ and $M$ cannot be compact,
and there exists a point
$$p\in M\setminus\big(A\cup C\big)\neq \emptyset.$$
Let $V$ be the open connected component of $M\setminus C$ that contains the point $p$.
$g$ is holomorphic on $V$ and vanishes in a neighborhood of $p$,
hence $g\equiv 0$ on $V$.
On the other hand
$V\cap D'$ is not empty since $C\subset\subset D'$ and $M$ is connected.
So, $g\equiv 0$ on $V\cap D'\neq\emptyset$,
and this implies by \eqref{eq:wtf} that
$\wt{F}\equiv F$ on $V\cap D'$ (which is not empty).
But
$$\emptyset \neq V\cap D'\ \subset\ D'\setminus C\ \subset\ D'\setminus K',$$
where $D'\setminus K'$ is connected,
and by the identity theorem we obtain
$$\wt{F}|_{D'\setminus K'} \equiv F$$
just as needed.

\section{Remarks on Takegoshi's Vanishing Theorem \ref{thm:takegoshi}}\label{sec:remarks}

Let $N$ be a complex manifold of dimension $n$ and $\Phi\in C^\infty(N)$
a real valued function on $N$. We denote by $H(\Phi)_p$ the complex Hessian
of $\Phi$ at $p\in N$, and set
$$\sigma(\Phi) = \max_{p\in N} \mbox{rank}\ H(\Phi)_p.$$
For $\Phi$ and $\Omega^n_N$, the sheaf of germs of holomorphic $n$-forms on $N$,
we denote by $A(p,q)$ the following assertion:
$\sigma(\Phi)=n-p+1$ and $H^q(N,\Omega^n_N)$, $H^{q+1}(N,\Omega^n_N)$ are Hausdorff.
$N$ is called weakly $1$-complete if it possesses a $C^\infty$ plurisubharmonic
exhaustion function. The key point in Takegoshi's vanishing theorem is the following
result which is proved by use of K\"ahler identities and a priori $L^2$ estimates 
for the $\dq$-operator (\cite{Ta}, Theorem 2.1):

\begin{thm}\label{thm:kaehler}
Let $N$ be a weakly $1$-complete K\"ahler manifold with a \linebreak $C^\infty$-plurisubharmonic
exhaustion function $\Phi$. Suppose with respect to $\Omega^n_N$ and $\Phi$ that
$A(p,q)$ is true for $q\geq p\geq 1$. Then: $H^q(N,\Omega^n_N)\cong H^{n-q}_{cpt}(N,\OO)=0$.
\end{thm}

Now, for the proof of Theorem \ref{thm:takegoshi}, a crucial question is 
how to get a K\"ahler metric involved if $\pi: M\rightarrow X$ is a proper modification,
but $M$ not necessary of K\"ahler type. This problem can be settled by
Hironaka's Chow lemma (see \cite{Hi}):

\begin{thm}\label{thm:hironaka}
Let $\pi: M \rightarrow X$ be a proper modification of complex spaces, $X$ reduced.
Then there exists a proper modification $\pi': M'\rightarrow X$ which is a locally finite (with respect to $X$)
sequence of blow-ups and a holomorphic map $h: M'\rightarrow M$ such that
$\pi'=\pi\circ h$.
\end{thm}

Note that $h: M'\rightarrow M$ is a proper modification, too.\\

%\newpage
%~\\[-11mm]

Now, let $x\in X$. Then $x$ has a Stein neighborhood $V$ with a $C^\infty$ strictly plurisubharmonic exhaustion function $\varphi$
such that $V':=\pi'^{-1}(V)$ is of K\"ahler type and has a $C^\infty$ plurisubharmonic exhaustion function $\varphi\circ \pi'$.
Note that the $H^q(V',\Omega^n_{M'})$ are Hausdorff for all $q\geq 0$ because $V'$ is holomorphically convex
(see \cite{La}, Theorem 2.1, and \cite{Pr}, Lemma II.1, and use the Remmert reduction). So, $H^q(V',\Omega^n_{M'})=0$ for $q\geq 1$
by Theorem \ref{thm:kaehler}. A similar reasoning shows that $R^q h_* \Omega^n_{M'}=0$ (cf. \cite{Ta}, paragraph 3)
and this yields
\begin{eqnarray}\label{eq:last}
H^q(\pi^{-1}(V),\Omega^n_M) \cong H^q(V',h_* \Omega^n_M) = H^q(V',\Omega^n_{M'}) = 0
\end{eqnarray}
because $h_* \Omega^n_{M'}=\Omega^n_M$.
\eqref{eq:last} gives Takegoshi's Theorem \ref{thm:takegoshi}.

%\newpage
\vspace{6mm}
{\bf Acknowledgments}

\vspace{2mm}
This work was done while the author was visiting the Texas A\&M University in College Station,
supported by a fellowship within the Postdoc-Programme of the German Academic Exchange Service (DAAD).
The author would like to thank the SCV group at the Texas A\&M University for its hospitality.

\end{document}